\documentclass[12pt]{article}
\usepackage{mathrsfs}
\usepackage{amsfonts}
\usepackage{amsthm,amsmath,amssymb,anysize,longtable}

\newtheorem{theorem}{Theorem}
\newtheorem{lemma}[theorem]{Lemma}

\newtheorem{definition}[theorem]{Definition}

\usepackage{xcolor}

\usepackage[numbers,sort&compress]{natbib}
\usepackage{amsmath}
\allowdisplaybreaks[3]
\marginsize{4.2cm}{4.2cm}{3.6cm}{3cm}
\numberwithin{equation}{section}

\author{\textsc{Xuelian Guo and Liming Tang\footnote{corresponding author}}\\
\small{School of Mathematical Sciences}\\
\small{Harbin Normal University}\\
\small{150025 Harbin, China}\\
\small{E-mail: limingtang@hrbnu.edu.cn}}

\date{ }
\date{ }
\usepackage[symbol]{footmisc}

\begin{document}

\thispagestyle{empty}

\noindent{\Large
Maps on the mirror Heisenberg-Virasoro algebra, II}
 \footnote{
The first part of the work is supported by the NNSF of China (Nos: 12001141,11971134) and NSF of Hei Longjiang Province (No. JQ2020A002);
FCT   UIDB/00212/2020 and UIDP/00212/2020.
The second part of this work is supported by the Russian Science Foundation under grant 22-71-10001.
}

	\bigskip
	
	 \bigskip

\begin{center}	
	{\bf
		
    Xuelian Guo\footnote{School of Mathematical Sciences, Harbin Normal University, 150025 Harbin, China;\ guoxuelian@stu.hrbnu.edu.cn},
    Ivan Kaygorodov\footnote{CMA-UBI, Universidade da Beira Interior, Covilh\~{a}, Portugal;  Moscow Center for Fundamental and Applied Mathematics,      Russia;  Saint Petersburg  University, Russia;\ kaygorodov.ivan@gmail.com} \&
Liming Tang\footnote{School of Mathematical Sciences, Harbin Normal University, 150025 Harbin, China; \ limingtang@hrbnu.edu.cn}\footnote{corresponding author}}
\end{center}

 \begin{quotation}
{\small\noindent \textbf{Abstract}:
This is the second paper in our series of papers dedicated to the study of maps on the mirror Heisenberg-Virasoro algebra. The first paper is dedicated to the study of unary maps and the present paper is dedicated to the study of binary maps.
Namely, we describe
biderivations and left-symmetric algebra structures
 on the complex mirror Heisenberg-Virasoro algebra.

\medskip
 \vspace{0.05cm} \noindent{\textbf{Keywords}}:
 Lie algebra; left-symmetric algebra; biderivation

\medskip

\vspace{0.05cm} \noindent \textbf{Mathematics Subject Classification
2020}: 17B40, 17B65, 17B68}
\end{quotation}
 \medskip

\section*{Introduction}

One of the classical pure algebraic problems is the study of linear and bilinear maps on a certain algebra (see, for example, \cite{I19,PR}, survey \cite{k23}, and references therein).
The most famous examples of binary maps on algebras are biderivations. The study of biderivations on a Lie algebra can be traced back to the study of commutative maps on associative algebras. In 2009, D. Benkovič  described all biderivations on triangular algebras \cite{BT}.
Y. Q. Du and Y. Wang determined biderivations on the generalized matrix algebra in 2013  \cite{DY}.
Also in 2013, D. Y. Wang, and X. X. Yu characterized biderivations and linear commuting maps on the Schr\"{o}dinger-Virasoro algebra   \cite{WD}.
In 2017, X. Cheng, M. Wang, J. Sun, and H. Zhang described biderivations and linear commuting maps on the Lie algebra $\mathfrak{gca}$ \cite{CX}.
In 2017, biderivations, post-Lie algebra structures, and linear commutative maps on $W$-algebras were studied by X. M. Tang   \cite{XB}.
In 2018, X. W. Liu, X. Q. Guo, and K. M. Zhao described biderivations on the block Lie algebra  \cite{LX}.
In the same year, X. M. Tang studied biderivations and commutative post-Lie algebra structures on Lie algebras ${\mathcal W}(a, b)$  \cite{BC}.
Later,  X. M. Tang and X. T. Li determined biderivations on the twisted Heisenberg-Virasoro algebra   \cite{TX}.
Biderivations play an important role in describing Poisson structures on associative algebras \cite{kk}.
In 2024, H. J. Liu and Z. X. Chen determined all biderivations on the centreless mirror Heisenberg-Virasoro algebra through the action of the biderivations on the basis and then described the biderivations on the mirror Heisenberg-Virasoro algebra. As an application, the description of commutative linear maps was given \cite{CL}.
In the present paper, by using the isomorphism relationship between the mirror Heisenberg-Virasoro algebra and ${\mathcal W}(\frac12, 0),$ all biderivations on the mirror Heisenberg-Virasoro algebra are determined. Some applications of the biderivations are given: commutative linear maps and commutative post-Lie algebra structures on the mirror Heisenberg-Virasoro algebra, and biderivations on the graded mirror Heisenberg-Virasoro left-symmetric algebra are described.

\smallskip

The left-symmetric algebras or LSAs for short, play an important role in mathematics and physics. In 1857, A. Cayley first introduced left-symmetric algebras in the context of the root algebra \cite{AO3}. With the passage of time, E. B. Vinberg, J.-L. Koszul, and M. Gerstenhaber again introduced   left symmetric algebras in the 1960s.
LSAs also went by many different names, being called Vinberg algebra,  Koszul algebra, or pre-Lie algebra, among others. Later, LSAs has been widely concerned in mathematics, and its research has made great progress. It is closely related to vertex algebra, and it is also important in the affine theory  \cite{EB3,BD3,BD21}. The left symmetric algebraic structure on a Lie algebra originates from the theory of affine manifold  \cite{BD1}.
In 2011,  X. L. Kong,  H. J. Chen, and C. M. Bai determined the left-symmetric algebra structure on Witt algebra and Virasoro algebra  \cite{WV}.
In 2012, H. J. Chen and J. B. Li studied left-symmetric algebra structures on $W (2, 2)$  \cite{LS}. They also determined the left-symmetric algebra structure on the twisted Heisenberg-Virasoro algebra in 2014   \cite{CH}.
In 2022, C. K. Xu determined compatible left-symmetric algebra structures on the high-rank Witt algebra and Virasoro algebra (see \cite{XC1}).
In the present paper, we described compatible left-symmetric algebra structures on the mirror Heisenberg-Virasoro algebra.

\smallskip

In this paper, we study biderivations and compatible left-symmetric algebra structures on the complex mirror Heisenberg-Virasoro algebra ({\rm mHV} algebra).
It is a Lie algebra with a basis
$\{d_{m}, h_{r}, {\bf c, l}\ |\ m\in \mathbb{Z}, r\in \frac{1}{2}+ \mathbb{Z}\}$
 and the following multiplication table (as usual, all zero multiplications  will be omitted):
\[[d_{m}, d_{n}]=(m-n)d_{m+n}+\frac{m^{3}-m}{12}\delta_{m+n}, _{0}{\bf c},\
[d_{m}, h_{r}]=-rh_{m+r},\
[h_{r}, h_{s}]=r\delta_{r+s,0}{\bf l},\]
where  $m, n\in \mathbb{Z}$, $r,s\in \frac{1}{2}+\mathbb{Z}$.
The mirror Heisenberg-Virasoro algebra is a $\mathbb{Z}$-graded algebra.
Virasoro algebra is a subalgebra of {\rm mHV} algebra.
 The mirror  Heisenberg-Virasoro algebra,  whose structure is similar to that of the twisted Heisenberg-Virasoro algebra, is the even part of mirror $N = 2$ superconformal algebra \cite{BK}. As it is well known, representation theory is essential in the research of Lie algebras. The Whittaker modules and the tensor products of Whittaker modules on {\rm mHV}  algebra, which are non-weight modules, are determined. Sufficient and necessary conditions are given for these non-weight modules on {\rm mHV}  algebra to be irreducible  \cite{WE}. Also, the tensor product weight modules on {\rm mHV}  algebra are studied, and some examples of irreducible weight modules on {\rm mHV}  algebra are given in \cite{EH}. All Harish-Chandra modules on {\rm mHV}  algebra are classified in \cite{ET}.
 Derivations, $2$-local derivations and $\frac{1}{2}$-derivations on  {\rm mHV}  algebra are classified in \cite{ES,GKT}.

\smallskip

The structure of this paper is as follows.
Firstly, by using the known biderivations on ${\mathcal W}(\frac12,0)$, all biderivations on the mirror Heisenberg-Virasoro algebra are determined.
Secondly, its applications are given. One application is different from the method proved by M. Liang, J. Lou, and S. L. Gao  \cite{MH}, the commutative linear map is determined by the definition of the biderivation. Another application is to prove that post-Lie algebra structures on the mirror Heisenberg-Virasoro algebra are trivial.
Third, the relationship between the coefficient functions of the graded conditions required to satisfy the mirror Heisenberg-Virasoro algebra without a center is determined. These relationships are used to determine concrete expressions. Finally,   the left-symmetric algebra structure on the mirror Heisenberg-Virasoro algebra is determined.
 In the end, it is proved that biderivations on the left-symmetric algebra of the mirror Heisenberg-Virasoro algebra are trivial.

\section{Biderivations on {\rm mHV} algebra}

 \begin{definition}
     Let $(\mathfrak L, [\cdot,\cdot])$ be a Lie algebra and
     $f$ is a bilinear map on $\mathfrak L$.
     If $f$  is a derivation of both components, that is, for every $x\in \mathfrak L$, there exist linear maps $\phi_x,\ \psi_x$ of $\mathfrak L$ to itself,  such that $\phi_x=f(x,\cdot),\ \psi_x=f(\cdot, x)$ are   derivations on $\mathfrak L$, i.e.
\begin{longtable}{lcl}
  $f([x, y], z)$&$=$&$[f(x,z), y]+[x,  f(y,z)],$\\
  $f(x, [y, z])$&$=$&$[f(x,y), z]+[y, f(x, z)].$
\end{longtable}
\noindent Then $f$ is called a biderivation of $\mathfrak L$.
  If  $f(x,y)=\lambda[x, y]$ for an element $\lambda \in \mathbb C$, then $f$ is an inner biderivation.
 Denote by ${\rm B}(\mathfrak L)$ the set of all biderivations on $\mathfrak L$.
 \end{definition}

\begin{lemma}(see, \cite{TX})\label{4.2}
Let $\mathfrak L$ be a perfect Lie algebra and $f$ is a biderivation of $\mathfrak L$. If $\alpha \in Z(\mathfrak  L)$, then  $f(x, \alpha)=f(\alpha, x)=0$ for all $x\in\mathfrak  L$.
\end{lemma}

\begin{lemma}(see, \cite{TX})
  Let $\mathfrak  L$ be a perfect Lie algebra and $Z'(\mathfrak  L)\subset Z(\mathfrak
 L)$, then a  linear map
     $\pi\colon {\rm B}(\mathfrak L)\rightarrow {\rm B}(\mathfrak L/Z'(\mathfrak L)),$
defined by
\begin{center}$\pi(f)(x+Z'(\mathfrak L), y+Z'(\mathfrak L))= f(x, y)+Z'(\mathfrak L)$
 \end{center} is injective.
 Moreover, $\pi$ is bijective if and only if a biderivation of $\mathfrak L/Z'(\mathfrak L)$ can be extended to a biderivation of $\mathfrak L$.
\end{lemma}

Let $\mathfrak L=\widetilde{\mathfrak L}\oplus Z'(\mathfrak L)$, then $\widetilde{\mathfrak
 L}\cong \mathfrak L/Z'(\mathfrak L)$.
 As follows from the two lemmas above, if $\mathfrak L$ is perfect,
 for describing all biderivations on $\mathfrak L$,
 we  need to draw all biderivations on $\widetilde{\mathfrak L}$.

\subsection{Biderivations on {\rm mHV} algebra}

 Let $\mathfrak{D}$ be the complex mirror Heisenberg-Virasoro algebra
 (to simplify, we will fix this notation for the rest of the paper).
 Then $\mathfrak{D}$ is perfect,
 $\mathfrak{D}=\widetilde{\mathfrak{D}}\oplus Z(\mathfrak{D})$, where $Z(\mathfrak{D})= \langle \bf c, \bf l \rangle$, $\widetilde{\mathfrak{D}}$ is a Lie algebra with a basis
 $\{d_{n}, h_{n+\frac{1}{2}}\ |\ n\in \mathbb{Z}\}$ and the following multiplication table:
 \[
  [d_m, d_n]=(m-n)d_{m+n},\ [d_m, h_{n+\frac12}]=-(n+\frac{1}{2})h_{m+n+\frac{1}{2}}.
 \]

By virtue of \cite{GS}, $\widetilde{\mathfrak{D}}={\rm W}'(0,0)\cong {\mathcal W}(\frac{1}{2},0)$,  where ${\mathcal W}(\frac{1}{2},0)= \langle d_n, h_n\ |\ n\in\mathbb{Z}\rangle$, and satisfies
  \begin{center}
 $[d_m, d_n]=(m-n)d_{m+n},\ [d_m, h_n]=-(n+\frac{1}{2})h_{m+n}.$
  \end{center}

\begin{lemma}(see, \cite{BC})
Let $f$ be a biderivation of ${\mathcal W}(\frac{1}{2},0)$, then any $x, y\in{\mathcal W}(\frac{1}{2},0)$, $f(x, y)=\lambda[x, y]+\widetilde{\Upsilon}_{\Omega}(x,y)$, and $\widetilde{\Upsilon}_{\Omega}$ satisfies
\begin{longtable}{lcl}
  $\widetilde{\Upsilon}_{\Omega}(d_{m},d_{n})$&$=$&$\sum_{k\in\mathbb{Z}}(k+\frac{1}{2})\mu_{k}h_{m+n+k},$\\
  $\widetilde{\Upsilon}_{\Omega}(d_{m},h_{n})$&$=$&$\widetilde{\Upsilon}_{\Omega}(h_{m},d_{n})\ =\ \widetilde{\Upsilon}_{\Omega}(h_{m},h_{n})\ = \ 0,$
\end{longtable}
where $\lambda\in\mathbb{C}$ and $\Omega=(\mu_{k})_{k\in\mathbb{Z}}$ which contains only finitely many
nonzero entries.
\end{lemma}

The last observation trivially gives the key statement of the present section.

\begin{lemma}\label{4.5}
 Let  $\tilde{f}$ be a biderivation of $\widetilde{\mathfrak{D}}$, then  we have $\tilde{f}(x, y)=\lambda[x, y]+\Upsilon_{\Omega}(x,y)$,  and $\Upsilon_{\Omega}$ satisfies
 \begin{longtable}{lcl}
    $\Upsilon_{\Omega}(d_{m},d_{n})$&$=$&$\sum_{k\in\mathbb{Z}}(k+\frac{1}{2})\mu_{k}h_{m+n+k+\frac{1}{2}},$\\
  $\Upsilon_{\Omega}(d_{m},h_{n+\frac{1}{2}})$&$=$&$\Upsilon_{\Omega}(h_{m+\frac{1}{2}},d_{n})
  =\Upsilon_{\Omega}(h_{m+\frac{1}{2}},h_{n+\frac{1}{2}})=0,$
  \end{longtable}
  where $\lambda\in\mathbb{C}$ and $\Omega=(\mu_{k})_{k\in\mathbb{Z}}$ which contains only finitely many
nonzero entries.

\end{lemma}

\begin{theorem}\label{4.6}
  A bilinear map  $f$ is  a biderivation of $\mathfrak{D}$ if and only if there exists $\lambda\in\mathbb{C}$, such that for any $x,y\in\mathfrak{D}$, $f(x, y)=\lambda[x, y]+\Upsilon_{\Omega}(x,y)$, and $\Upsilon_{\Omega}$ satisfies
 \begin{longtable}{lcl}
    $\Upsilon_{\Omega}(d_{m},d_{n})$ &$ =$&$\sum_{k\in\mathbb{Z}}(k+\frac{1}{2})\mu_{k}h_{m+n+k+\frac{1}{2}},$\\
  $\Upsilon_{\Omega}(d_{m},h_{n+\frac{1}{2}})$ & $=$&$\Upsilon_{\Omega}(h_{m+\frac{1}{2}},d_{n})
\  =\ \Upsilon_{\Omega}(h_{m+\frac{1}{2}},h_{n+\frac{1}{2}})\ =\ 0,$
  \end{longtable}
  where $\Omega=(\mu_{k})_{k\in\mathbb{Z}}$ which contains only finitely many
nonzero entries.
 \begin{proof}
   Adequacy is clearly established, and necessity is proved below.
   Let $f$ be a biderivation of $\mathfrak{D}$, we define a linear map
   $\pi\colon {\rm B}(\mathfrak{D})\rightarrow {\rm B}(\mathfrak{D}/Z(\mathfrak{D})),$
   that satisfies
   \[\pi(f)(x+Z(\mathfrak{D}), y+Z(\mathfrak{D}))=f(x, y)+Z(\mathfrak{D}),\]
   where $Z(\mathfrak{D})=  \langle {\bf c,l}\rangle$. Notice that $\mathfrak{D}/Z(\mathfrak{D})\cong\widetilde{\mathfrak{D}}$, then by virtue of Lemma \ref{4.5}, put
   \[f(x, y)=\lambda[x, y]+\Upsilon_{\Omega}(x,y), \]
    where $\lambda\in\mathbb{C}$ and $ x, y\in \langle d_i, h_{i+\frac{1}{2}}\ |\ i\in\mathbb{Z}\rangle$.
    According to lemma \ref{4.2}, for
    $\alpha,  \beta\in
    \langle {\bf c,l}\rangle,$ we have
    \begin{longtable}{lcl}
      $f(x+\alpha, y+\beta)$&$=$&$f(x, y)\ =\ \lambda[x, y]+\Upsilon_{\Omega}(x,y)\ = \ $ \\
                          &$=$ &$\lambda[x+\alpha, y+\beta]+\Upsilon_{\Omega}(x+\alpha, y+\beta).$
    \end{longtable}
    Therefore, the conclusion holds.
 \end{proof}
\end{theorem}

\subsection{Applications of   biderivations on {\rm mHV} algebra}

\begin{definition}
   A linear  map $\phi$ on a Lie algebra $\mathfrak L$ is called
   a commuting map, if  $[\phi(x), x]=0$.
\end{definition}

If $\phi$ is a linear commuting map, then $[\phi(x), y]=[x, \phi(y)]$.
Put $f(x, y)=[\phi(x), y]=[x, \phi(y)]$, then $f$ is a biderivation of $\mathfrak  L$.

\begin{theorem}
A linear map $\phi$ on $\mathfrak{D}$ is commutative
if and only if there exists $\lambda\in\mathbb{C}$ and a linear map  $\tau\colon
  \mathfrak{D}\rightarrow Z(\mathfrak{D})$, such that $\phi(x)=\lambda x+\tau(x).$

  \begin{proof}
    Sufficiency:
$[\phi(x), x]=[\lambda x+\tau(x), x]=0.$
    Hence, $\phi$ is a linear commuting map.
    Conversely, let $f(x, y)= [\phi(x), y]$, then $f$ is a biderivation of  $\mathfrak{D}$. According to Theorem \ref{4.6}, we have
    \[
    f(x, y)=[\phi(x), y]=\lambda [x, y]+\Upsilon_{\Omega}(x,y).
    \]

  \noindent  Since $f(x,x)=0$, then $f(x,y)=-f(y,x)$, i.e.
    \[\lambda [x, y]+\Upsilon_{\Omega}(x,y)=-\lambda [y,x]-\Upsilon_{\Omega}(y,x).\]
    Therefore, $\Upsilon_{\Omega}(x,y)=0.$
    So    $[\phi(x)-\lambda x, y]=0.$
    Hence,\  $\phi(x)-\lambda x\in Z(\mathfrak{D})$. So there exists  $\tau\colon
  \mathfrak{D}\rightarrow Z(\mathfrak{D})$, such that
  $\phi(x)=\lambda x+\tau(x).$
  \end{proof}
\end{theorem}

\begin{definition}
Let $(\mathfrak L, [\cdot,\cdot])$ be a  Lie algebra.
A commutative post-Lie algebra structure on $\mathfrak L$ is a
bilinear product $x\cdot y$ on $\mathfrak L$ satisfying the following identities:
\begin{longtable}{rcl}
  $x\cdot y$&$=$&$y\cdot x,$\\
  $[x,y]\cdot z$&$=$&$x\cdot (y\cdot z)-y\cdot (x\cdot z),$\\
  $x\cdot[y,z]$& $=$&$[x\cdot y, z]+[y, x\cdot z].$
\end{longtable}
\noindent  A post-Lie algebra $(\mathfrak L, [\cdot ,\cdot], \cdot)$ is said to be trivial if $x\cdot y=0$.
  If a bilinear map $f\colon \mathfrak L\times \mathfrak L\rightarrow \mathfrak  L$ is defined by $f(x, y)= x\cdot y$, then $f$ is a biderivation of $\mathfrak  L$.
\end{definition}

\begin{theorem}
  Any commutative post-Lie algebra structure on  $\mathfrak{D}$ is trivial.
  \begin{proof}
    Suppose that $(\mathfrak{D}, [\cdot,\cdot], \cdot)$ is a commutative post-Lie algebra. By Lemma \ref{4.6}, there exists $\lambda \in \mathbb{C}$, such that
    \[
    f(x, y)=x\cdot y=\lambda[x, y]+\Upsilon_{\Omega}(x,y).
    \]

   \noindent Since $x\cdot y=y\cdot x$, then
    \begin{center}$\lambda[x, y]+\Upsilon_{\Omega}(x,y)\ =\ \lambda[y,x]+\Upsilon_{\Omega}(y,x) $ \ and \
    $2\lambda[x, y]=0.$
    Therefore, $\lambda=0$.
    \end{center}

\noindent    If there is $\mu_{k}\in \Omega$ that makes $\mu_{k}\neq0$, then for $d_{1}, d_{2}, d_{3}$,
    \begin{longtable}{lcl}
      $[d_{2},d_{1}]\cdot d_{3}$&$= $&$d_{2}\cdot(d_{1}\cdot d_{3})-d_{1}\cdot(d_{2}\cdot d_{3})$\\

                              &$=$&$d_{2}\cdot(\sum_{k\in\mathbb{Z}}(k+\frac{1}{2})\mu_{k}h_{k+\frac{9}{2}})
                              -d_{1}\cdot(\sum_{k\in\mathbb{Z}}(k+\frac{1}{2})\mu_{k}h_{k+\frac{11}{2}})=0,$ \\
    $ [d_2, d_1]\cdot d_3 $&$=$&$d_3\cdot d_3=\sum_{k\in\mathbb{Z}}(k+\frac12)\mu_kh_{k+\frac{13}{2}}.$
    \end{longtable}
  \noindent Comparing the obtained results, we have  $\Upsilon_\Omega=0.$ Hence, $f(x, y)= x\cdot y=0$.    Therefore, the conclusion holds.
  \end{proof}
\end{theorem}

\section{The left-symmetric algebra structures on  {\rm mHV} algebra }

\begin{definition}\label{5.1}
 An algebra  ${\rm A}$ is called   a left-symmetric algebra if it satisfies
 \[(xy)z-x(yz)=(yx)z-y(xz).\]
 \end{definition}
%It is well known that left-symmetric algebras are Lie-admissible algebras.
    %For any $x\in\rm{A}$, denote by $L_x$ the left multiplication operator, that is, $L_x(y)=xy,\,y\in\rm{A}$.
 It is known that  the commutator multiplication $[x, y]=xy-yx,$
 defined on a left-symmetric algebra $\rm{A},$
    gives a Lie algebra $\mathcal{L}(\rm{A})$, which is called the sub-adjacent of $\rm{A}$ and $\rm{A}$ is called a compatible left-symmetric algebra structure on the Lie algebra $\mathcal{L}(\rm{A})$.
   A compatible left-symmetric algebra structure on $\mathfrak{L}$ will be denoted by $\mathcal{A}(\mathfrak{L}).$
A compatible left-symmetric algebra structure on the Virasoro algebra $\mathcal{V}=\langle d_n,  {\bf c} \ |\  n \in \mathbb Z\rangle$
is said to
have the natural grading condition if   nonzero multiplications of basis elements of $\mathcal{A}(\mathcal{V})$ satisfies
\begin{eqnarray}
d_m d_n=f(m, n)d_{m+n}+\omega(m, n){\bf c} \label{1}.
\end{eqnarray}
\begin{lemma}\label{t5.2}(see, \cite{WV})
 Any compatible  left-symmetric algebra structure on the Virasoro algebra $\mathcal{V}$ satisfying (\ref{1}) is given by the multiplication table
  \[
  d_m d_n=-\frac{n(1+\epsilon n)}{1+\epsilon(m+n)}d_{m+n}+\frac{1}{24}(m^3-m+(\epsilon-\epsilon^{-1})m^2)\delta_{m+n,0}{\bf c},
  \]
  where $\epsilon \neq 0$ and $\epsilon^{-1} \notin \mathbb Z.$
\end{lemma}
 The mirror Heisenberg-Virasoro algebra $\mathfrak{D}$ contains twisted Heisenberg subalgebra and Virasoro subalgebra and is also $\mathbb{Z}$-graded.
   Let us remember the grading components of $\mathfrak{D}:$
\begin{longtable}{rclrclrcl}
$\mathfrak{D}_{0}$&$=$&$\langle d_{0},h_{\frac{1}{2}},{\bf c}\rangle,$&
$\mathfrak{D}_{-1}$&$=$&$\langle d_{-1},h_{-\frac{1}{2}}, {\bf l}\rangle,$&
$\mathfrak{D}_{n}$&$=$&$\langle d_{n},h_{n+\frac{1}{2}}\ | \ n\in \mathbb{Z}\backslash\{0,-1\}\rangle.$
\end{longtable}
 Naturally, we have the natural grading condition
(in agreement with the natural grading condition of the multiplication of $\mathcal{A}(\mathcal{V})$)
if   nonzero multiplications of basis elements of  $\mathcal{A}(\mathfrak{D})$ satisfies the following relations, which we will denote by ($\diamond$):

 \begin{longtable}{rcl}
   $d_m d_n$&$=$&$f(m, n)d_{m+n}+\omega(m, n){\bf c},$ \\
   $d_m h_{n+\frac{1}{2}}$&$=$&$g(m, n)h_{m+n+\frac12},$\\
   $h_{m+\frac{1}{2}} d_n $&$=$&$h(m, n)h_{m+n+\frac12},$\\
   $h_{m+\frac12}h_{n+\frac12}$&$=$&$a(m, n)d_{m+n}+b(m, n)h_{m+n+\frac12}+\rho(m, n){\bf l}.$\\
 %  &{\bf cl}={\bf lc}={\bf c}d_m=d_m {\bf c}={\bf l}d_{m}=d_{m}{\bf l}=0,\\
  % &{\bf c}h_{m+\frac12}=h_{m+\frac12}{\bf c}={\bf l}h_{m+\frac12}=h_{m+\frac12}{\bf l}=0
 \end{longtable}
 The main aim of the present section is to prove the following theorem.
 \begin{theorem}\label{T5.6}
 Any compatible left-symmetric algebra structure on the mirror Heisenberg-Virasoro algebra satisfying relation $(\diamond)$ is given by  the following multiplication table:
\begin{longtable}{rcl}
   $d_m d_n$&$=$&$-\frac{n(1+\epsilon n)}{1+\epsilon(m+n)} d_{m+n}+\frac{1}{24}(m^3-m+(\epsilon-\epsilon^{-1})m^2)\delta_{m+n,0}{\bf c},$\\
   $d_m h_{n+\frac{1}{2}}$&$=$&$ -(n+\frac12) h_{m+n+\frac12},$\\
   $h_{m+\frac12}h_{n+\frac12}$&$=$&$\frac12(n+\frac12)\delta_{n+m+1,0}{\bf l},$\\
   \end{longtable}
   where $\epsilon \neq 0$ and $\epsilon^{-1} \notin \mathbb Z.$
\end{theorem}

The algebra, obtained in the previous theorem will be called
the graded mirror Heisenberg-Virasoro left-symmetric algebra $\mathcal{A}$.

\subsection{Left-symmetric algebra structures on {\rm mHV} algebra.}
 \begin{lemma}
  A bilinear product induced from   $(\diamond)$  gives a compatible left-symmetric algebra structure on $\widetilde{\mathfrak{D}}$ if and only if
   \begin{longtable}{rcl}
     $f(m, n)-f(n, m)$&$=$&$m-n,$\\
     $g(m, n)-h(n, m)$&$=$&$-(n+\frac12),$\\
     $a(m, n)$&$=$&$a(n, m),$ \\
     $b(m, n)$&$=$&$b(n, m),$\\
     $f(m, k)f(n, m+k)-f(n, k)f(m, n+k)$&$=$&$(n-m)f(m+n, k),$\\
     $g(m, k)g(n, m+k)-g(n, k)g(m, n+k)$&$=$&$(n-m)g(m+n, k),$\\
     $h(m, k)g(n, m+k)-f(n, k)h(m, n+k)$&$=$&$-(m+\frac12)h(m+n, k),$\\
     $a(m, k)f(n, m+k)-g(n, k)a(m, n+k)$&$=$&$-(m+\frac12)a(m+n, k),$ \\
     $b(m, k)g(n, m+k)-g(n, k)b(m, n+k)$&$=$&$-(m+\frac12)b(m+n, k),$\\
     $h(m, k)a(a, m+k)-h(n, k)a(m, n+k)$&$=$&$0,$\\
     $h(m, k)b(n, m+k)-h(n, k)b(m, n+k)$&$=$&$0,$\\
     $b(m, k)a(n, m+k)-b(m, k)a(m, n+k)$&$=$&$0,$\\
     $a(m, k)h(n, m+k)+b(m, k)b(n, m+k)$&$=$&$a(n, k)h(m, n+k)+b(n, k)b(m, n+k).$
   \end{longtable}
The present relations will be denoted by $(\star)$.
   \begin{proof}
     \begin{itemize}
   \item[(i)] On the one hand
    $[d_{m}, d_{n}]=(m-n)d_{m+n},$
    and on the other hand
    \begin{center}
      $[d_{m}, d_{n}] \ =\ d_md_n-d_nd_m\ = \ f(m, n)d_{m+n}-f(n, m)d_{m+n}.$
    \end{center}
    Comparing the two equations above,
    $f(m, n)-f(n, m)=m-n.$

    \item[(ii)] On the one hand
    $[d_m, h_{n+\frac12}]=-(n+\frac12)h_{m+n+\frac12},$
     and on the other hand
    \[ [d_m, h_{n+\frac12}]=g(m, n)h_{m+n+\frac12}-h(n, m)h_{m+n+\frac12}.\]
     Comparing the two equations above,
    $g(m, n)-h(n, m)=-(n+\frac12).$

    \item[(iii)] On the one hand
    $[h_{m+\frac12}, h_{n+\frac12}]=0,$ and
 on the other hand
    \[ [h_{m+\frac12}, h_{n+\frac12}]=a(m, n)d_{m+n}+b(m, n)h_{m+n+\frac12}-a(n, m)d_{m+n}-b(n, m)h_{m+n+\frac12}.
    \]
     Comparing the two equations above,
    $a(m, n)=a(n, m)$ and $b(m, n)=b(n, m).$

   \item[(iv)] By the left-symmetric identity $(d_n, d_m, d_k)=(d_m, d_n, d_k),$ we have
   \begin{flushleft}   $f(n, m)f(n+m, k)d_{n+m+k}- f(m, k)f(n, m+k)d_{m+n+k}  =$\end{flushleft}
      \begin{flushright}$= f(m, n)f(n+m, k)d_{n+m+k}- f(n, k)f(m, n+k)d_{m+n+k}.$
    \end{flushright}
    Hence, $f(m, k)f(n, m+k)-f(n, k)f(m, n+k)=(n-m)f(m+n, k).$
  \end{itemize}
    \item[(v)]
    By the left-symmetric identity $(d_n, d_m, h_{k+\frac12})=(d_m, d_n, h_{k+\frac12})$, we have
   \begin{flushleft}
  $f(n, m)g(n+m, k)h_{m+n+k+\frac12}-g(m, k)g(n, m+k)h_{m+n+k+\frac12}
   =$\end{flushleft}
   \begin{flushright}$= (m, n)g(n+m, k)h_{m+n+k+\frac12}-g(n, k)g(m, n+k)h_{m+n+k+\frac12}.$
  \end{flushright}
  Hence, $ g(m, k)g(n, m+k)-g(n, k)g(m, n+k)=(n-m)g(m+n, k).$

    \item[(vi)]  By the left-symmetric identity $(d_n, h_{m+\frac12}, d_k)=(h_{m+\frac12}, d_n, d_k)$, we have
   \begin{flushleft}$g(n, m)h(m+n, k)h_{m+n+k+\frac12}-h(m, k)g(n, m+k)h_{m+n+k+\frac12}=$\end{flushleft}
  \begin{flushright}   $=h(m, n)h(m+n, k)h_{m+n+k+\frac12}-f(n, k)h(m, n+k)h_{m+n+k+\frac12}.$\end{flushright}
Hence, $ h(m, k)g(n, m+k)-f(n, k)h(m, n+k)=-(m+\frac12)h(m+n, k).$

    \item[(vii)] By the left-symmetric identity $(d_n, h_{m+\frac12}, h_{k+\frac12})=(h_{m+\frac12}, d_n, h_{k+\frac12})$, we have
  \begin{flushleft}$g(n, m)\big(a(n+m, k)d_{m+n+k}+b(n+m, k)h_{m+n+k+\frac12}\big)$\end{flushleft}
 \begin{flushright}     $-a(m, k)f(n, m+k)d_{m+n+k}-b(m, k)g(n, m+k)h_{m+n+k+\frac12}=$\end{flushright}

   \begin{flushleft}$=h(m, n)\big(a(n+m, k)d_{m+n+k}+b(n+m, k)h_{m+n+k+\frac12}\big)$\end{flushleft}
  \begin{flushright}    $-g(n, k)\big(a(m, n+k)d_{m+n+k}+b(m, n+k)h_{m+n+k+\frac12}\big).$\end{flushright}

 Hence,
 \begin{longtable}{rcl}
     $a(m, k)f(n, m+k)-g(n, k)a(m, n+k)$&$=$&$-(m+\frac12)a(m+n, k),$\\
     $b(m, k)g(n, m+k)-g(n, k)b(m, n+k)$&$=$&$-(m+\frac12)b(m+n, k).$
\end{longtable}

    \item[(viii)]
    By the left-symmetric identity $(h_{n+\frac12}, h_{m+\frac12}, d_{k})=(h_{m+\frac12}, h_{n+\frac12}, d_k)$, we have
   \begin{flushleft}$h(m, k)\big(a(n, m+k)d_{m+n+k}+b(n, m+k)h_{m+n+k+\frac12}\big)=$\end{flushleft}
       \begin{flushright}    $h(n, k)\big(a(m, n+k)d_{m+n+k}+b(m, n+k)h_{m+n+k+\frac12}\big).$\end{flushright}

Hence,
     $h(m, k)a(a, m+k)=h(n, k)a(m, n+k)$ and
     $h(m, k)b(n, m+k)=h(n, k)b(m, n+k).$

   \item[(ix)]   By the left-symmetric identity $(h_{n+\frac12}, h_{m+\frac12}, h_{k+\frac12})=(h_{m+\frac12}, h_{n+\frac12}, h_{k+\frac12})$, we have
        \begin{flushleft}$a(m, k)h(n, m+k)h_{m+n+k+\frac12}+b(m, k)\big(a(n, m+k)d_{m+n+k}+b(n, m+k)h_{m+n+k+\frac12}\big)=$\end{flushleft}
           \begin{flushright}    $a(n, k)h(m, n+k)h_{m+n+k+\frac12}+b(n, k)\big(a(m, n+k)d_{m+n+k}+b(m, n+k)h_{m+n+k+\frac12}\big).$\end{flushright}
Hence, \begin{longtable}{lcl}
         $b(m, k)a(n, m+k)-b(m, k)a(m, n+k)$&$=$&$0,$\\
         $a(m, k)h(n, m+k)+b(m, k)b(n, m+k)$&$=$&$a(n, k)h(m, n+k)+b(n, k)b(m, n+k).$
       \end{longtable}

   Therefore, the conclusion holds.
   \end{proof}
 \end{lemma}

 \begin{lemma}
  A multiplication defined by relations $(\diamond)$ gives a compatible left-symmetric algebra structure on $\mathfrak{D}$ if and only if $(\star)$ and the following identities hold for all $m, n, k\in\mathbb{Z}$,
  \begin{longtable}{rcl}
    $\omega(m, n)-\omega(n, m)$&$=$&$\frac{m^3-m}{12}\delta_{m+n,0},$\\
    $\rho(m, n)-\rho(n, m)$&$=$&$(m+\frac12)\delta_{m+n+1, 0},$\\
    $f(n, k)\omega(m, {n+k})-f(m, k)\omega(n, m+k)$&$=$&$(m-n)\omega(m+n, k),$ \\
    $g(m, k)\rho(n, m+k)-(n+\frac12)\rho(m+n, k)$&$=$&$0,$\\
    $a(m, k)\omega(n, {m+k})$&$=$&$0,$\\
    $h(n, k)\rho(m, n+k)$&$=$&$h(m, k)\rho(n, m+k),$\\
    $b(m, k)\rho(n, m+k )$&$=$&$b(n, k)\rho( m, n+k).$
  \end{longtable}

The present relations will be denoted by $(\ast)$.
\end{lemma}

  \begin{proof}
 \begin{enumerate}

 \item[{\rm (i)}] On the one hand,
$ [d_{m}, d_{n}]=(m-n)d_{m+n}+\frac{m^{3}-m}{12}\delta_{m+n}, _{0}{\bf c},$
    and
    on the other hand,
    \begin{center}
      $[d_{m}, d_{n}]=d_md_n-d_nd_m
                    =f(m, n)d_{m+n}+\omega(m, n){\bf c}-f(n, m)d_{m+n}-\omega(n, m){\bf c}.$
    \end{center}
    Comparing the two above formulas, we obtain
    $\omega(m, n)-\omega(n, m)=\frac{m^3-m}{12}\delta_{m+n,0}.$

    \item[{\rm (ii)}] On the one hand,
    $[h_{m+\frac12}, h_{n+\frac12}]=(m+\frac12)\delta_{m+n+1,0}{\bf l},$
and   on the other hand,
    \[ [h_{m+\frac12}, h_{n+\frac12}]=\rho(m, n){\bf l}-\rho(n, m){\bf l}.
    \]
    Comparing the two above formulas, we obtain
    \[\rho(m, n)-\rho(n, m)=(m+\frac12)\delta_{m+n+1, 0}.\]

    \item[{\rm (iii)}]     By the left-symmetric identity  $(d_m, d_n, d_k)=(d_n, d_m, d_k)$, we have
    \begin{longtable}{rcl}
      $f(n, m)\big(f(n+m, k)d_{n+m+k}$&$+$&$\omega({n+m}, k){\bf c}\big)-$\\
      &$-$&$f(m, k)\big(f(n, m+k)d_{m+n+k}+\omega(n, {m+k}){\bf c}\big)
      =$\\
      $f(m, n)\big(f(n+m, k)d_{n+m+k}$&$+$&$\omega({n+m}, k){\bf c}\big)-$\\
      &$-$&$f(n, k)\big(f(m, n+k)d_{m+n+k}+\omega(m, {n+k}){\bf c}\big).$
    \end{longtable}
    Hence, \[f(n, k)\omega(m, {n+k})-f(m, k)\omega(n, {m+k})=(m-n)\omega({n+m}, k).\]

    \item[{\rm (iv)}]     By the left-symmetric identity  $(d_n, h_{m+\frac12}, h_{k+\frac12})=(h_{m+\frac12}, d_n, h_{k+\frac12})$, we have
  \begin{flushleft}
      $g(n, m)\big(a(n+m, k)d_{m+n+k}+b(n+m, k)h_{m+n+k+\frac12}+\rho(m+n, k){\bf l}\big) -$
  \end{flushleft}
\begin{flushright}
    $-a(m, k)\big(f(n, m+k)d_{m+n+k}+\omega(n, {m+k}){\bf c}\big)-b(m, k)g(n, m+k)h_{m+n+k+\frac12}\ =$\end{flushright}
\begin{flushleft}
         $h(m, n)\big(a(n+m, k)d_{m+n+k}+b(n+m, k)h_{m+n+k+\frac12}+\rho(m+n, k){\bf l}\big) - $
\end{flushleft}
\begin{flushright}
    $-g(n, k)\big(a(m, n+k)d_{m+n+k}+b(m, n+k)h_{m+n+k+\frac12}+\rho(m, n+k){\bf l}\big).$\end{flushright}
   Therefore,
   \begin{center}
   $g(n, k)\rho(m, n+k)=(m+\frac12)\rho(m+n, k),$ and $
     a(m, k)\omega(n, {m+k})=0.$\end{center}

  \item[{\rm (v)}]     By the left-symmetric identity  $(h_{n+\frac12}, h_{m+\frac12}, d_{k})=(h_{m+\frac12}, h_{n+\frac12}, d_k)$, we obtain
\begin{flushleft}
    $h(m, k)\big(a(n, m+k)d_{m+n+k}+b(n, m+k)h_{m+n+k+\frac12}+\rho(n, m+k){\bf l}\big)\ =$
\end{flushleft}
\begin{flushright}
         $= \ h(n, k)\big(a(m, n+k)d_{m+n+k}+b(m, n+k)h_{m+n+k+\frac12}+\rho(m, n+k){\bf l}\big).$
         \end{flushright}
   Therefore, \[
     h(n, k)\rho(m, n+k)=h(m, k)\rho(n, m+k).
   \]

   \item[{\rm (vi)}]     By the left-symmetric identity  $(h_{n+\frac12}, h_{m+\frac12}, h_{k+\frac12})=(h_{m+\frac12}, h_{n+\frac12}, h_{k+\frac12})$, we have
       \begin{flushleft}
         $b(m, k)\big(a(n, m+k)d_{m+n+k}+b(n, m+k)h_{m+n+k+\frac12}+\rho(n, m+k){\bf l}\big) \ =$ \end{flushleft}
\begin{flushright}
    $= \ b(n, k)\big(a(m, n+k)d_{m+n+k}+b(m, n+k)h_{m+n+k+\frac12}+\rho(m, n+k){\bf l}\big).$
\end{flushright}
       Hence,
       $       b(n, k)\rho(m, n+k)=b(m, k)\rho(n, m+k).$

       \end{enumerate}
   Therefore, the conclusion holds.
 \end{proof}
   By Theorem \ref{t5.2}, in order to obtain a compatible left-symmetric algebra structure on $\mathfrak{D}$ with the grading condition ($\diamond$), one can suppose
   \begin{eqnarray}
   f(m, n)=-\frac{n(1+\epsilon n)}{1+\epsilon(m+n)}  \label{5}.
   \end{eqnarray}
Furthermore, the following theorem determines all compatible left-symmetric algebraic structures on $\mathfrak{D}$ that satisfies $f(m,n)$ defined by formula (\ref{5}).
 \begin{theorem}\label{t5.4}
 Any compatible left-symmetric algebraic structure on the mirror Heisenberg-Virasoro algebra  satisfying the relation $(\diamond)$ is given by the following functions:
  \begin{longtable}{rclrcl}
  $f(m, n) $&$=$&$-\frac{n(1+\epsilon n)}{1+\epsilon(m+n)},$ &
  $ a(m, n)$&$=$&$0,$\\

  $g(m, n)$&$=$&$-(n+\frac12),$ &
  $b(m, n)$&$=$&$0,$ \\
    $\rho(n, m) $&$=$&$\frac12(n+\frac12)\delta_{n+m+1,0},$ &
      $h(m, n)$&$=$&$0,$\\

    $\omega(m, n)$&$=$&\multicolumn{4}{l}{$\frac{1}{24}(m^3-m+(\epsilon-\epsilon^{-1})m^2)\delta_{m+n,0},$}\\
  \end{longtable}
  where $\epsilon\neq 0$ and $\epsilon^{-1} \notin \mathbb{Z}.$
%There is only the following solution simultaneously satisfying all equations in (\ref{5.4})\:
% \begin{longtable}{lcllcl}
%  $h(m, n)$&$=$&$0, $ &
%  $g(m, n)$&$=$&$ -(n+\frac12),$\\
%  $a(m, n)$&$=$&$0,$&$  b(m, n)$&$=$&$0.$
% \end{longtable}
 \begin{proof}
 For any $m, n\in\mathbb{Z}$, let
 \[
 G(m, n)=\frac{1+\epsilon(m+n)}{1+\epsilon n}g(m, n),\quad H(n, m)=\frac{1+\epsilon(m+n)}{1+\epsilon m}h(n, m),
 \]
 then by ($\star$), we have
 \begin{eqnarray}
   (1+\epsilon n)G(m, n)-(1+\epsilon m)H(n, m)=-(n+\frac12)(1+\epsilon(m+n)), \label{6}\\
   G(n, k)G(m, n+k)-G(m, k)G(n, m+k)=(m-n)G(m+n, k), \label{7}\\
   H(n, k)G(m, n+k)+kH(n, m+k)=-(n+\frac12)H(m+n, k). \label{8}
 \end{eqnarray}
 Taking $m=0$ in (\ref{6}) and (\ref{8}), there is
 \begin{longtable}{lcl}
   $(1+\epsilon n)G(0, n)-H(n,0)$&$=$&$-(n+\frac12)(1+\epsilon n),$\\
   $H(n, k)G(0, n+k)+kH(n, k)$&$=$&$-(n+\frac12)H(n, k).$
 \end{longtable}
 \noindent Taking $n=n+k$ in the first formula, then
 \[
 G(0, n+k)=\frac{H(n+k,0)}{1+\epsilon(n+k)}-(n+k+\frac12).
 \]
 Then it is plugged into the second formula, and we obtain
 \[
 \frac{H(n, k)H(n+k,0)}{1+\epsilon(n+k)}-(n+k+\frac12)H(n, k)+kH(n, k)=-(n+\frac12)H(n, k).
 \]
 Therefore, $H(n, k)H(n+k,0)=0$. Taking $k=0$, then $H(n,0)^2=0$, i.e.
 \begin{center}
     $H(n, 0)=0$ and $G(0, n)=-(n+\frac12).$
 \end{center}
  Taking $k=0, m=-n$ in (\ref{7}), then
 \begin{eqnarray}
 G(n, 0)G(-n, n)-G(-n, 0)G(n, -n)=-2nG(0,0)=n.  \label{9}
 \end{eqnarray}
And by (\ref{6}), we have
\begin{eqnarray}
H(n, m)=\frac{1+\epsilon n}{1+\epsilon m}G(m, n)+\frac{(n+\frac12)(1+\epsilon(m+n))}{1+\epsilon m}. \label{10}
\end{eqnarray}
 Taking $n=0, m=-k$ in (\ref{8}), then
$H(0, k)G(-k, k)+kH(0,0)=-\frac12H(-k, k).$
Hence, $H(-k, k)=-2H(0, k)G(-k, k)$. On the one hand, by virtue of (\ref{10}),
\[
H(-k, k)=\frac{1-\epsilon k}{1+\epsilon k}G(k, -k)+\frac{-k+\frac12}{1+\epsilon k}.
\]
On the other hand, by (\ref{9}) and (\ref{10}), we have
\[
G(k,0)G(-k, k)=G(-k,0)G(k, -k)+k,\quad H(0, k)=\frac{G(k,0)}{1+\epsilon k}+\frac12.
\]
So
\begin{longtable}{lcl}
  $H(-k, k)$&$=$&$-2H(0, k)G(-k, k) \ =\ -2\left(\frac{G(k,0)}{1+\epsilon k}+\frac12\right)G(-k, k) = $\\

          & $=$ & $\frac{-2G(k,0)G(-k, k)}{1+\epsilon k}-G(-k, k)\
          = \ \frac{-2(G(-k,0)G(k, -k)+k)}{1+\epsilon k}-G(-k, k) \ = $\\
          &$=$ &$\frac{-2G(-k,0)G(k, -k)}{1+\epsilon k}-\frac{2k}{1+\epsilon k}-G(-k, k).$
\end{longtable}
\noindent Comparing the above two formulas, we obtain
\begin{center}
$(1-\epsilon k)G(k,-k)+(-k+\frac12)=-2G(-k,0)G(k,-k)-2k-(1+\epsilon k)G(-k, k).$
\end{center}
Therefore,
\begin{eqnarray}
(1+2g(-k,0))g(k,-k)+g(-k, k)=-(k+\frac12). \label{11}
\end{eqnarray}
According to $G(0, n)=-(n+\frac12)$, $g(0, n)=-(n+\frac12)$. Therefore, put $$g(m, n)=mg'(m, n)-(n+\frac12).$$ Then
\begin{longtable}{lcllcl}
  $h(n, m)$&$=$&$mg'(m, n),$ &
  $g(-k, 0)$&$=$&$-kg'(-k,0)-\frac12,$\\
  $g(k, -k)$&$=$&$k(g'(k, -k)+1)-\frac12,$ &
  $g(-k, k)$&$=$&$-k(g'(-k, k)+1)-\frac12.$
\end{longtable}
Let  $ f(k)=g'(k, -k)+1,$ then $g(k, -k)=kf(k)-\frac12. $
It is plugged into (\ref{11}). Then
\[
kg'(-k,0)(-2kf(k)+1)=k(f(-k)-1).
\]

Let us define two important subsets of $\mathbb Z:$
\begin{longtable}{lcl}

$A$ & $=$ & $\{ m \ | \  f(m)=\frac{1}{2m} \mbox{ \ and \ }f(-m)=1 \},$\\
$B$ & $=$ & $\{ m \ | \ f(m)=1   \mbox{ \ and \ }  f(m)\neq -\frac{1}{2m}\}.$\\
\end{longtable}
It is easy to see that  $A \cap -A = \varnothing.$
Let $A=\{m_0, m_1, \ldots, m_i, \ldots \}.$
Since~$g(0, 0)=-\frac12$,~then~$m_i\neq 0.$~
Namely, if $m_i \in A \cap -A ,$ then
\begin{center}
$f(m_i)=\frac{1}{2m_i},$ \ $f(-m_i)=1,$ \
$f(-m_i)=-\frac{1}{2m_i},$ \ $f(m_i)=1.$\end{center}

Obviously, this is a contradiction. In this case,
\begin{longtable}{clclc}
  $g(m_i, -m_i)$&$=$&$m_if(m_i)-\frac12 $&$=$&$0,$\\
  $g(-m_i, m_i)$&$=$&$-m_if(-m_i)-\frac12$&$=$&$-(m_i+\frac12),$
\end{longtable}
and
\begin{longtable}{lcllcl}
  $h(m_i, -m_i)$&$=$&$0, $& \ \
  $h(-m_i, m_i)$&$=$&$-m_i+\frac12.$
\end{longtable}
Inside of the proof of our Theorem,
we highlight two following lemmas,
which gives the characterization of sets $A$ and $B$.
Namely, we have to prove that $A=\varnothing$ and $B=\mathbb Z.$

\begin{lemma}
$A=\varnothing.$
\end{lemma}

\begin{proof}
    Set $k\in\mathbb{Z}\backslash\{-A \cup B\}$, there is
\begin{longtable}{lcllcl}
 $h(0, k) $&$= $&$\frac{k(f(k)-1)}{2kf(-k)+1},$ &
 $h(-k, k)$&$=$&$k(f(k)-1),$\\
 $h(k, -k) $&$=$&$-k(f(-k)-1),$ &
 $h(k, 0)$&$=$&$0,$
 \end{longtable}
 and
 \begin{longtable}{lcllcllcl}
 $g'(k, 0)$&$=$&$\frac{f(k)-1}{2kf(-k)+1},$&
   %g'(0, k)&=\frac{f(k)-1}{2kf(-k)+1},\\
   $g'(-k, k)$&$=$&$f(-k)-1,$&
   $g'(k, -k)$&$=$&$f(k)-1.$
 \end{longtable}
 \noindent Therefore, by the first formula, and then set~$$g'(m, n)=\frac{g''(m, n)(f(m)-1)}{2mf(-m)+1},  \  m\in\mathbb{Z}\backslash\{-A \cup B\},$$
in this case,
 \begin{longtable}{lcllcl}
   $g''(m, 0)$&$=$&$1,$ &
  $g'(k, -k)$&$=$&$\frac{g''(k, -k)(f(k)-1)}{2kf(-k)+1}.$
 \end{longtable}
\noindent The last gives,
\begin{center}
    $g''(k, -k)=2kf(-k)+1, \  k\in\mathbb{Z}\backslash\{-A \cup B\}.$
    \end{center}

 \noindent Then by $g''(m, 0)=1$, set~$g''(m, n)=ng'''(m, n)+1$.
 Therefore, $$g''(k, -k)=-kg'''(k, -k)+1.$$ So $g'''(k, -k)=-2f(-k),$ and $g'''(m_i, -m_i)=-2f(-m_i)=-2.$
Hence, for~$m\in\mathbb{Z}$ and $ n\in\mathbb{Z}\backslash\{-A \cup B\}, $ we have
\[ h(m, n)=\frac{n(f(n)-1)(mg'''(n, m)+1)}{2nf(-n)+1}.\]
Therefore, \[ g(n, m)= \frac{n(f(n)-1)(mg'''(n, m)+1)}{2nf(-n)+1}-(m+\frac12).\]

If~$mg'''(n, m)+1\equiv0$, then~$g'''(n, m)=-\frac1m$. Therefore, $g'''(m_i, -m_i)=\frac1m_i$, which is contradict to~$g'''(m_i, -m_i)=-2f(-m_i)=-2$, then there are only two cases:

\begin{enumerate}
    \item[I.]
\ ``$mg'''(n, m)+1\equiv0$''  and~$A=\varnothing$. Then
$h(m, n)=0, ~g(m, n)=-(n+\frac12).$

\item[II.] \ ``$mg'''(n, m)+1$''~is not constant to~0.

 If~$A\neq\varnothing$~, then at this point, $g(m, n)=mg'(m, n)-(n+\frac12), m, n\in\mathbb{Z}$. Taking $m=0$ in the fourth formula of ($\ast$), then there is
     \[
     (n+k+1)\rho(n, k)=0.
     \]

     When\,$n+k+1\neq0$, $\rho(n, k)=0$. Let $\rho(n, k)=\theta(n)\delta_{n+k+1,0}$ and put it into the fourth formula of ($\ast$):
      \[g(m, k)\theta(n)\delta_{n+m+k+1,0}-(n+\frac12)\theta(m+n)\delta_{n+m+k+1,0}=0.\]

      When $m\in\mathbb{Z}\backslash\{-A \cup B\}$, according to the previous proof,
      \[ g(m, k)= \frac{m(f(m)-1)(kg'''(m, k)+1)}{2mf(-m)+1}-(k+\frac12),\]
      then the above equation becomes
      \begin{center}
          $\Big( \big(\frac{m(f(m)-1)(kg'''(m, k)+1)}{2mf(-m)+1}-(k+\frac12)\big)\theta(n)-(n+\frac12)\theta(m+n) \Big)\delta_{n+m+k+1,0}=0. $     \end{center}
      Taking~$m+n+k+1=0,$ where $m\in\mathbb{Z}\backslash\{-A \cup B\}$,
      then the above equation becomes
      \begin{center}
      $\big(\frac{m(f(m)-1)(kg'''(m, k)+1)}{2mf(-m)+1}-(k+\frac12)\big)\theta(n)=(n+\frac12)\theta(m+n).$\end{center}
      Making~$n=0$,~by $m+k+1=0,$ $k=-1-m$, where   $m\in\mathbb{Z}\backslash\{-A \cup B\}$, then the above equation becomes
      \begin{center}
          $ \frac{m(f(m)-1)}{2mf(-m)+1}\big((-1-m)g'''(m, -1-m)+1\big)\theta(0)-\frac12\theta(m)=-(m+\frac12)\theta(0).$
      \end{center}
      Therefore, when~$m=m_i$, then the above equation becomes
     \begin{eqnarray}
       \frac{m_i(f(m_i)-1)}{2m_if(-m_i)+1}\big((-1-m_i)g'''(m_i, -1-m_i)+1\big)\theta(0)-\frac12\theta(m_i)=-(m_i+\frac12)\theta(0) \label{16}.
     \end{eqnarray}

 In the sixth formula of ($\ast$), set   $k\in\mathbb{Z}\backslash\{-A \cup B\}$, and from the previous proof, the sixth formula changes to:
    \begin{center}
    $\frac{k(f(k)-1)(ng'''(k, n)+1)}{2kf(-k)+1}\theta(m)\delta_{m+n+k+1, 0}=\frac{k(f(k)-1)(mg'''(k, m)+1)}{2kf(-k)+1}\theta(n)\delta_{m+n+k+1, 0}.$
    \end{center}

 Taking~$m+n+k+1=0$, then the above equation becomes
      \begin{center}
          $(ng'''(k, n)+1)\theta(m)=(mg'''(k, m)+1)\theta(n).$
      \end{center}
      Let~$n=0$. Then $k=-1-m$. Since~ $k\in\mathbb{Z}\backslash\{-A \cup B\}$, then
        \begin{center}$\theta(m)=(mg'''(k, m)+1)\theta(0).$  \end{center}
     Therefore, when~$k=m_i$,
      \begin{center}
          $\theta(-1-m_i)=((-1-m_i)g'''(m_i, -1-m_i)+1)\theta(0).$
      \end{center}

     Therefore,  the equation~(\ref{16}) becomes
     \begin{eqnarray}
       \frac{m_i(f(m_i)-1)}{2m_if(-m_i)+1}\theta(-1-m_i)-\frac12\theta(m_i)=-(m_i+\frac12)\theta(0) \label{17}.
     \end{eqnarray}

       Place~$\rho(n, k)=\theta(n)\delta_{n+k+1,0}$~into the second formula of ($\ast$). Then there is
       \begin{center}
           $\theta(m)\delta_{m+n+1, 0}-\theta(n)\delta_{m+n+1, 0}=(m+\frac12)\delta_{m+n+1, 0},$
       \end{center}
       taking ~$m+n+1=0$, then~$n=-1-m$, so  the above equation becomes
        \begin{eqnarray}
\theta(-1-m)-\theta(m)=- \Big(m+\frac12 \Big), \ m\in\mathbb{Z}.\label{theta(-1-m)}
        \end{eqnarray}
         In particular, taking~$m=m_i$~and~$m=0$,~then
      \begin{eqnarray}
        \theta(-1-m_i)-\theta(m_i)=-\Big(m_i+\frac12 \Big) \label{18},
      \end{eqnarray}
      and $\theta(-1)=\theta(0)-\frac12.$

     From the equation~(\ref{17})~and~(\ref{18}), there is
      \begin{eqnarray}
       \Big(\frac{m_i(f(m_i)-1)}{2m_if(-m_i)+1}-\frac12\Big)\theta(-1-m_i)=
       \Big(\frac12-\theta(0)\Big) \Big(m_i+\frac12 \Big) \label{19}.
     \end{eqnarray}
      By~$f(m_i)=\frac{1}{2m_i}$~~$f(-m_i)=1$, then
      \begin{eqnarray}
        \theta(-1-m_i)=-{m_i^{-1}}{\Big(m_i+\frac12\Big)^2}\Big(\frac12-\theta(0)\Big) \label{20}.
      \end{eqnarray}

     From the fourth formula of ($\ast$), there is
     \[g(m, k)\theta(n)\delta_{m+n+k+1, 0}-(n+\frac12)\theta(m+n)\delta_{m+n+k+1, 0}=0.\]
      Taking~$m+n+k+1=0$,~then
      $g(m, k)\theta(n)=(n+\frac12)\theta(m+n),$
     and by~$g(-m_i, m_i)=-m_i+\frac12$, taking~$m=-m_i, k=m_i, n=-1$,~then the above equation becomes
  \begin{center}
  $g(-m_i, m_i)\theta(-1)=-\frac12\theta(-1-m_i),$
  \end{center}
     that is
     \begin{center}
         $-(m_i+\frac12)(\theta(0)-\frac12)=-\frac12\theta(-1-m_i).$
         \end{center}
      Therefore,
     \begin{eqnarray}
       \theta(-1-m_i)=-2\Big(m_i+\frac12 \Big)\Big(\frac12-\theta(0)\Big)\label{21}.
     \end{eqnarray}

    If~$\theta(0)=\frac12$,
then by (\ref{theta(-1-m)}), we have  $\theta(-1)=0$.
On the other hand, in the fourth formula of~($\ast$),  for $m=1,$ $k=1, n=-1$,~we have
    $\theta(0)=-2g(1, -1)\theta(-1)=0.$
The last gives a contradiction and~$\theta(0)\neq\frac12.$
Then comparing the equation (\ref{20}) with (\ref{21}),
we have $m_i=\frac 12.$ But on the other hand, $m_i \in \mathbb Z.$
So~$A=\varnothing$.
%{\red Therefore, if $f(m)$ is not constant to 1, then $A$ must be an empty set.}

\end{enumerate}
\end{proof}

\begin{lemma}
    $B=\mathbb Z.$

\end{lemma}

\begin{proof}
If $\mathbb{Z} \neq B$ and~$A=\varnothing$, then
 for~$m\in\mathbb{Z}\backslash B$,
 the formula~(\ref{17}), (\ref{18})~and (\ref{19}) hold. For formula~(\ref{19}), there is
       \begin{center}
       $\Big(\frac{m(f(m)-1)}{2mf(-m)+1}-\frac12\Big)\theta(-1-m)=
       \Big(\frac12-\theta(0)\Big)\Big(m+\frac12\Big).$
     \end{center}
     Then
    \begin{eqnarray}
      \theta(-1-m)=\frac{(m+\frac12)(2mf(-m)+1)}{m(f(m)-f(-m))-(m+\frac12)} \Big(\frac12-\theta(0)\Big), \    \forall~ m\in\mathbb{Z}\backslash B. \label{22}
    \end{eqnarray}

\noindent      By the fourth formula of ($\ast$), there is
     \[g(m, k)\theta(n)\delta_{m+n+k+1, 0}-\Big(n+\frac12 \Big)\theta(m+n)\delta_{m+n+k+1, 0}=0.\]
     Taking~$m+n+k+1=0$, and taking~$m=-m, k=m, n=-1$, then
      \begin{center}
          $g(-m, m)\theta(-1)=-\frac12\theta(-1-m).$
      \end{center}
      By  $g(-m, m)=-mf(-m)-\frac12$, then the above equation becomes
      \begin{eqnarray}
        \theta(-1-m)=-(2mf(-m)+1) \Big(\frac12-\theta(0)\Big),  \ \forall~m\in\mathbb{Z}\backslash B \label{23}.
      \end{eqnarray}
      Comparing~the equation (\ref{22})~and~(\ref{23}), there is
\begin{center}
    ${m+\frac12}{}=-(m(f(m)-f(-m))-(m+\frac12)),  \  \forall~m\in\mathbb{Z}\backslash  B.$
\end{center}
Obviously, if $m\in\mathbb{Z}\backslash B,$~$f(m)\neq1$, there must be~$f(m)=f(-m)$.
It is also easy to see that for all $m \in \mathbb Z,$ we have $f(m)=f(-m)$.
Then
     \begin{center}
         $\theta(-1-m)=-(2mf(m)+1)\Big(\frac12-\theta(0)\Big),\ \forall~m\in\mathbb{Z}\backslash B.$
     \end{center}
Similarly to the previous discussion of~$\theta(0)$, we have that  $\theta(0)\neq\frac12$. Then
\begin{center}
    $ \theta(-1-m)=(2mf(m)+1)\theta(-1), \ \forall~m\in\mathbb{Z}\backslash B.$
\end{center}
At this time, for~$n\in\mathbb{Z}\backslash B,$ we have
\begin{longtable}{lcllcllcl}
  $h(-n, n)$&$=$&$n(f(n)-1),$ & \
  $g(n, -n)$&$=$&$nf(n)-\frac12,$& \
  $g(n, 0)$&$=$&$-\frac{n+\frac12}{2nf(n)+1}.$
\end{longtable}
In the seventh formula of~($\star$), taking~$m=n,$ $k=-n,$ $n=n,$ $n\in\mathbb{Z}\backslash B$,~then
\begin{center}
$h(n, -n)g(n, 0)-f(n, -n)h(n, 0)=-\Big(n+\frac12 \Big)h(2n, -n).$\end{center}
Therefore,
\begin{longtable}{lcl}
$h(2n, -n) $&$=$&$-\frac{n(f(n)-1)}{2nf(n)+1},$\\
$g(-n, 2n)$&$=$&$-\frac{n(f(n)-1)}{2nf(n)+1}-(2n+\frac12).$
\end{longtable}

\noindent In the seventh formula of~($\star$), taking~$m=2n, k=-n, n=-n, n\in\mathbb{Z}\backslash B$,~then
\begin{center}
$h(2n, -n)g(-n, n)-f(-n, -n)h(2n, -2n)=-(2n+\frac12)h(n, -n).$
\end{center}
By calculating,
\begin{center}
   $-\frac{1-\epsilon n}{1-2\epsilon n}h(2n, -2n)=2n(f(n)-1).$
\end{center}

If~$2n\in B$,~then~$h(2n, -2n)=0$.~While~$f(n)\neq1$,~so the two sides of the above formula are not equal. Therefore, $2n\in\mathbb{Z}\backslash B$. Then the above formula becomes
\begin{eqnarray}
  \frac{(f(2n)-1)(1-\epsilon n)}{1-2\epsilon n}=f(n)-1.  \label{24}
\end{eqnarray}
From this formula, $f(n)\neq0$.
In the fourth formula of~($\ast$), taking~$m=-n,$ $k=2n$ and $n=-1-n,$~then
\begin{center}
    $g(-n, 2n)\theta(-1-n)=-\Big(n+\frac12\Big)\theta(-1-2n).$
    \end{center}
By calculating,
\begin{align*}
  &\text{the left end of equation}=\Big(n+\frac12 \Big)(4nf(n)+1),\\
  &\text{the right end of equation}=\Big(n+\frac12 \Big)(4nf(2n)+1).
\end{align*}
Therefore, $f(n)=f(2n).$ And by~$\epsilon \neq 0$, then the two sides of equation (\ref{24}) are not equal.
 Therefore, for $m\in\mathbb{Z}, f(m)$ can only be equal to 1.

\end{proof}

{\bf Continuation of the proof of Theorem \ref{t5.4}.}
Now we have $A=\varnothing$ and $B=\mathbb Z.$
At this point, for~$k\in\mathbb{Z}$, there is
 \begin{center}
     $g(-k, k)=-(k+\frac12)$ and  $h(-k, k)=0.$
 \end{center}
 In the seventh formula of~($\star$), taking~$n=-m-k$, then
 \begin{center}
 $h(m, k)g(-m-k, m+k)-f(-m-k, k)h(m, -m)=-(m+\frac12)h(-k, k).$
 \end{center}
 Therefore, for~$m, k\in\mathbb{Z}$,~there is
 $h(m, k)=0$
 and
 $g(k, m)=-(m+\frac12).$

 Nextly, the expressions for $\rho(m, n)~(m, n\in\mathbb{Z})$ are determined.
Taking $m=0$ in the fourth equation of ($\ast$), we have
     $
     (n+k+1)\rho(n, k)=0.$
     When $n+k+1\neq0$, $\rho(n, k)=0$.
     It is plugged into the fifth equation of ($\ast$):
     \begin{center}
     $-\Big(k+\frac12\Big)\theta(n)\delta_{m+n+k+1,0}-\Big(n+\frac12\Big)\theta(m+n)\delta_{m+n+k+1,0}=0.
     $\end{center}
      Taking $m+n+k+1=0$, we have
    \begin{center}
        $(m+n+\frac12)\theta(n)=(n+\frac12)\theta(m+n).$
    \end{center}
    Let $n=1$, we have
    \begin{center}
    $\theta(m+1)=(\frac23 m+1)\theta(1)=(\frac23(m+1)+\frac13)\theta(1).$
    \end{center}
    Taking $m$ into $m-1$, there is
    \begin{center}
    $\theta(m)=\theta((m-1)+1)=\Big(\frac23((m-1)+1)+\frac13\Big)\theta(1)=(\frac23m+\frac13)\theta(1).$
    \end{center}
    It is plugged into the second equation of ($\ast$): taking $m+n+1=0$, we have
    \begin{center}
  $  \Big(\frac23 m+\frac13\Big)\theta(1)-\Big(\frac23 n+\frac13\Big)\theta(1)=m+\frac12.$
  \end{center}
    Hence, $\theta(1)=\frac34.$ Therefore, $\theta(n)=\frac12(n+\frac12).$ So
    \begin{center}
    $\rho(n, k)=\frac12(n+\frac12)\delta_{n+k+1,0}.$
    \end{center}

At last, the expressions for $a(m, n)$ and $b(m, n)$ are determined.
By the eighth equation of ($\star$), taking $n=0$, there is
\begin{center}
$-(m-\frac12)a(m, k)=-(m+\frac12)a(m, k).$
\end{center}
Therefore, $a(m, k)=0.$ Again, by the ninth equation of ($\star$), taking $n=0$, we have
\begin{center}
$-mb(m, k)=-(m+\frac12)b(m, k).$
\end{center}
Therefore, $b(m, k)=0.$

In turn, it is easy to verify that the function of the Theorem\,\ref{T5.6} satisfies ($\star$)  and ($\ast$).

   \end{proof}
\end{theorem}

\subsection{Applications}

Let us remember a well-known statement about biderivations of
sub-adjacent and original algebras.

\begin{lemma}\label{L4.10}
   Let $\mathcal{A}$ be an algebra, and
   $\mathfrak L(\mathcal{A})$ be a sub-adjacent Lie algebra of $\mathcal{A}$.
   Then every biderivation of $\mathcal{A}$ is also a biderivation.
 \end{lemma}

 %\begin{definition}
% A graded mirror Heisenberg-Virasoro left-symmetric algebra  is an algebra $\mathcal{A}$ with a basis $\{d_n, h_{n+\frac12},{\bf c, l}|n\in\mathbb{Z}\}$ such that
%  \begin{align*}
%    f&(m, n)=\frac{-n(1+\epsilon n)}{1+\epsilon(m+n)},\\
%  g&(m, n)=-(n+\frac12),\\
%  h&(m, n)=a(m, n)=b(m, n)=0,\\
%  \omega&(d_m, d_n)=\frac{1}{24}(m^3-m+(\epsilon-\epsilon^{-1})m^2)\delta_{m+n,0},\\
%  \omega&(h_{n+\frac12}, h_{m+\frac12})=\frac12(n+\frac12)\delta_{n+m+1,0},
%  \end{align*}
%  for any $m, n\in\mathbb{Z}, \epsilon\in\mathbb{C} $ and where {\rm Re}$\epsilon> 0$, $\epsilon^{-1} \notin \mathbb{Z}$\ or\ {\rm Re}$\epsilon=0$, {\rm Im}$\epsilon>0.$
% \end{definition}

\begin{theorem}
   Any biderivations $f$ of the graded mirror Heisenberg-Virasoro left-symmetric algebra $\mathcal{A}$ is trivial, i.e. $f=0.$
   \begin{proof}
     First, the mirror Heisenberg-Virasoro algebra $\mathfrak{D}$ is the sub-adjacent Lie algebra of the left-symmetric algebra $\mathcal{A}$.
     Let $f$ be a biderivation of graded mirror Heisenberg-Virasoro left-symmetric algebra $\mathcal{A}$. Then by Lemma \ref{L4.10},  $f$ is also a biderivation of \,$\mathfrak{D}$. According to Theorem \ref{4.6}, there exists $\lambda\in\mathbb{C}$, such that $f(x, y)=\lambda[x, y]+\Upsilon_\Omega(x, y)$.
     Hence,
     \begin{longtable}{lcl}
       $f(d_2 d_1, d_3)$&$=$&$\lambda [d_2 d_1, d_3]+\Upsilon_\Omega(d_2d_1, d_3) \ = \ $ \\

                      &$=$ &$\lambda\left[-\frac{1+\epsilon}{1+3\epsilon}d_{3}, d_3\right]-\Upsilon_\Omega(\frac{1+\epsilon}{1+3\epsilon}d_{3}, d_3) \ = \ -\frac{1+\epsilon}{1+3\epsilon}\sum_{k\in \mathbb{Z}}(k+\frac12)\mu_kh_{k+\frac{13}{2}}.$
     \end{longtable}
     On the other hand,
      \begin{longtable}{lcl}
       $f(d_2 d_1, d_3)$&$=$&$d_2 f(d_1, d_3)+f(d_2, d_3) d_1 \ = $\\

                      &$=$&$d_2(\lambda[d_1, d_3]+\Upsilon_\Omega(d_1, d_3))+\big(\lambda[d_2, d_3]+\Upsilon_\Omega(d_2, d_3)\big)d_1 \ =\ $\\

                      &$=$&$-2\lambda d_2d_{4}+d_2\Upsilon_\Omega(d_1, d_3)-\lambda d_{5}d_1+\Upsilon_\Omega(d_2, d_3)d_1 \ =\ $\\
                      &$=$&$-\lambda(2d_2d_4+d_5d_1)+\sum_{k\in \mathbb{Z}}(k+\frac12)\mu_kd_2h_{k+\frac{9}{2}}+\sum_{k\in \mathbb{Z}}(k+\frac12)\mu_kh_{k+\frac{9}{2}}d_1 \ = \ $\\

                      &$=$&$\lambda\left(\frac{8(1+3\epsilon)}{1+6\epsilon}+\frac{1+\epsilon}{1+6\epsilon}\right)d_{6}-\sum_{k\in \mathbb{Z}}(k+\frac12)(k+5)\mu_kh_{k+\frac{13}{2}}.$
     \end{longtable}
     Comparing the above two formulas, we have
    $\lambda=0$
     and
     $\mu_k=0,\ {\rm i.e.}\  \Upsilon_\Omega=0.$
  Hence, $f=0.$
   \end{proof}
\end{theorem}

\end{document}